\documentclass{amsart}
\usepackage{amssymb,amsthm}

\newtheorem{theorem}{Theorem}
\newtheorem{lemma}[theorem]{Lemma}
\newtheorem{corollary}[theorem]{Corollary}

\theoremstyle{definition}

\renewcommand{\P}{\mathcal P}
\newcommand{\B}{\mathcal B}

\newcommand{\ZZ}{\mathbb Z}
\newcommand{\DD}{\mathbb D}
\renewcommand{\SS}{\mathbb S}
\newcommand{\SD}{\DD_2}

\newcommand{\GL}{\mathrm{GL}}
\newcommand{\GenDih}{\mathrm{GenDih}}
\newcommand{\Aut}{\mathrm{Aut}}

\newcommand{\Sym}{\mathrm{Sym}}

\newcommand{\Cos}{\mathrm{Cos}}
\newcommand{\val}{\mathrm{val}}
\newcommand{\core}{\mathrm{core}}
\newcommand{\Pet}{\mathrm{Pet}}

\newcommand{\la}{\langle}
\newcommand{\ra}{\rangle}

\begin{document}
\title[Locally arc-transitive graphs of valence $\{3,4\}$]{Locally arc-transitive graphs of valence $\{3,4\}$ with trivial edge kernel}

\author[P. Poto\v{c}nik]{Primo\v{z} Poto\v{c}nik}
\address{Primo\v{z} Poto\v{c}nik,\newline
 Faculty of Mathematics and Physics,
 University of Ljubljana \newline 
Jadranska 19, 1000 Ljubljana, Slovenia}\email{primoz.potocnik@fmf.uni-lj.si}


\begin{abstract}
In this paper we consider connected locally  $G$-arc-transitive graphs with vertices of valence $3$ and $4$, such that 
the kernel $G_{uv}^{[1]}$ of the action of an edge-stabiliser on the neighourhood $\Gamma(u) \cup \Gamma(v)$
is trivial. We find nineteen finitely presented groups with the property that any
such group $G$ is a quotient of one of these groups.
As an application, we enumerate all connected locally arc-transitive graphs of valence $\{3,4\}$ on at most $350$ vertices
whose automorphism group contains a locally arc-transitive subgroup $G$ with
$G_{uv}^{[1]} = 1$.
\end{abstract}

\maketitle

\section{Introduction}
\label{sec:intro}
  An arc in a simple graph $\Gamma$ is an ordered pair of adjacent vertices of $\Gamma$.
  Let $\Gamma$ be graph and $G$ a group of automorphisms of $\Gamma$.
Then $\Gamma$ is said to be {\em $G$-arc-transitive} provided that $G$ acts transitively on the set of arcs of $\Gamma$.
Similarly, $\Gamma$ is said to be {\em locally $G$-arc-transitive} if for every vertex $v$
the stabiliser $G_v$ of $v$ acts transitively on the set of  all arcs of $\Gamma$ with the initial vertex being $v$.
A graph $\Gamma$ is arc-transitive if it is $\Aut(\Gamma)$-arc-transitive.  In this paper, we shall
be particularly interested in the structure of the vertex-stabilisers (and thus of the group $G$ itself) in certain locally $G$-arc-transitive graphs.
All the graphs in this paper are assumed to be connected.

If $\Gamma$ is a connected locally $G$-arc-transitive graph, then it is  well known that $G$ is transtiivie on the edges of $\Gamma$ 
and that it has at most two orbits on the vertex set $V(\Gamma)$. If $G$ is transitive on $V(\Gamma)$, then it is in fact arc-transitive.
 On the other hand, if $G$ has two orbits on $V(\Gamma)$, then we say that $\Gamma$ is {\em genuinely} locally $G$-arc-transitive. 
 In this case it can be shown that  $\Gamma$ is bipartite and that the two orbits form the bipartition of $\Gamma$. 
 Furthermore, the group $G$ is generated by a pair of stabilisers $G_u$ and $G_v$ of two adjacent vertices $u,v \in V(\Gamma)$.

If $G$ is an automorphism group of a graph $\Gamma$ and $v\in V(\Gamma)$, then we let $G_v^{\Gamma(v)}$ denote the permutation
group induced by the action of $G_v$ on the neighbourhood $\Gamma(v)$ of the vertex $v$, and let $G_v^{[1]}$ denote the kernel of this action
(that is, $G_v^{[1]}$ is the group of all those elements of $G$ that fix $v$ and each of its neighbours). 
Similarly, for an edge $uv$ of $\Gamma$ let $\Gamma(uv) = \Gamma(u) \cup \Gamma(v) \setminus \{u,v\}$ and let $G_{uv}^{\Gamma(uv)}$ denote
the permutation group induced by the action of $G_{uv}$ on $\Gamma(uv)$. Observe that the kernel of this action is the intersection $G_{u}^{[1]} \cap G_{v}^{[1]}$,
which will be denoted by $G_{uv}^{[1]}$ and called the {\em edge kernel} of the group $G$.
Note that $G_v^{\Gamma(v)} \cong G_v/G_v^{[1]}$ and $G_{uv}^{\Gamma(uv)} \cong G_{uv}/G_{uv}^{[1]}$.

Since in a locally $G$-arc-transitive graph $\Gamma$ the group $G$ has at most two orbits on $V(\Gamma)$ it follows that
the valence function can take only two values, say $d_1$ and $d_2$; the graph is then said to be 
{\em biregular} of valence $\{d_1,d_2\}$.

If one of the two valences, say $d_2$, is $2$, then it is easy to see that the graph can be obtained from
a $G$-arc-transitive graph of valence $d_1$ by subdividing each edge of that graph (where that graph
is allowed to have parallel edges in the case when the original graph contains cycles of length $4$).
In this sense, the case of genuinely locally
$G$-arc-transitive graphs of valence $\{2,d\}$ is equivalent to the family of $G$-arc-transitive graphs of valence $d$
and has, as such, received much attention in the past; in particular, the structure of the vertex-stabiliser $G_v$
has been determined for the cases $d=3$ (see  \cite{CL,DjM,tutte}) and $d=4$ and $5$ (see \cite{dj,pot,weiss}).

The case where $d_1=d_2=3$ was studied in several papers, most notably by Goldschmidt in \cite{Gold}, where it was proved that the group $G$
must be a quotient of one of the $15$ universal groups, and in \cite{CMMP}, where a complete list of
all such graphs on up to $768$ vertices was compiled. 

The purpose of this paper is to begin an investigation of the next interesting case where $\{d_1,d_2\} =\{ 3,4\}$.
In Section~\ref{sec:large}, we show that in this case the edge kernel $G_{uv}^{[1]}$ can be arbitrary large. This is in sharp contrast with
the behaviour of locally $G$-arc-transitive graphs of valence $\{3,3\}$, where the order of $G_{uv}^{[1]}$ divides $32$ (see \cite{CMMP,Gold}).

In the study of locally $G$-arc-transitive graphs, the case where the edge kernel $G_{uv}^{[1]}$ is not trivial is rather special 
(see for example \cite{vanBon}). In this paper, we shall restrict ourselves 
to the locally $G$-arc-transitive graphs of valence $\{3,4\}$ with
the trivial edge kernel. We will prove that, in this case, the group $G$ is a quotient of one of nineteen {\em universal} 
infinite finitely presented groups (that we give explicitly in terms of generators and relators).

\begin{theorem}
\label{the:main}
Let $\Gamma$ be a connected locally $G$-arc-transitive graph of valence $\{3,4\}$ and let $uv$ be an edge of $\Gamma$ with
$\val(v) =3$ and  $\val(u) = 4$.
If $G_{uv}^{[1]} = 1$,
then for some $i\in\{0,1,\ldots, 18\}$ there exists an epimorphism from the group $U_i$, given in Table~1, onto $G$, 
which maps the subgroups $L_i$, $B_i$ and $R_i$ isomorphically onto $G_v$, $G_{uv}$ and $G_u$, respectively.
\end{theorem}

\begin{center}
\begin{small} 
 Table 1: Universal groups for locally arc-transitive graphs of valence $\{3,4\}$ with $G_{uv}^{[1]}=1$.
 \end{small}
\\[2mm]
\begin{scriptsize}
\begin{tabular}{|c|}
\hline\hline
\begin{tabular}{l} 
 \\
           $U_0 =  \la a, c \mid a^3, c^4 \ra$ \\ \\ 
           $L_0 = \la a \ra \cong C_3, \quad B_0=1, \quad R_0=\la c \ra \cong C_4$ \\ \\
         \end{tabular} 
 \\
\hline
\begin{tabular}{l} 
 \\
           $U_1 =  \la a, x,y \mid a^3, x^2,y^2, [x,y] \ra$ \\ \\ 
           $L_1 = \la a \ra \cong C_3, \quad B_1=1, \quad R_1=\la x,y \ra \cong C_2\times C_2$ \\ \\
         \end{tabular} 
 \\
 \hline

 \begin{tabular}{l} 
 \\
           $U_2 = \la a,b,c \mid a^3, b^2, c^4,
                                                        [a,b],  (bc)^2 \ra$ \\ \\
           $L_2=\la a,b\ra   \cong C_6, \quad B_2=\la b \ra \cong C_2,\quad R_2=\la b,c\ra \cong D_4 $ \\ \\
         \end{tabular} 
 \\
 \hline

\begin{tabular}{l} 
 \\
           $U_3 = \la a,b,c \mid a^3, b^2, c^4,
                                                        (ab)^2,  (bc)^2 \ra$ \\ \\
           $L_3=\la a,b\ra  \cong S_3, \quad B_3=\la b \ra \cong C_2,\quad R_3=\la b,c\ra \cong D_4 $ \\ \\
         \end{tabular} 
 \\
\hline

\begin{tabular}{l} 
 \\
           $U_4 = \la a,c,x,y \mid a^9,c^3, x^2, y^2,
                                                        a^3=c, [x,y], x^c=y, y^c=xy \ra$ \\ \\
           $L_4=\la a ,c\ra \cong C_9, \quad B_4=\la c \ra \cong C_3,\quad R_4=\la c,x,y\ra \cong A_4 $ \\ \\
         \end{tabular} 
 \\
 \hline

\end{tabular}
\end{scriptsize}
\end{center}   

\newpage

\begin{center}
\begin{scriptsize}
\begin{tabular}{|c|}
\hline\hline

 \begin{tabular}{l} 
 \\
           $U_5 = \la a,c,x,y \mid a^3,c^3, x^2, y^2,
                                                        [a,c], [x,y], x^c=y, y^c=xy \ra$ \\ \\
           $L_5=\la a ,c\ra   \cong C_3 \times C_3, \quad B_5=\la c \ra \cong C_3,\quad R_5=\la c,x,y\ra \cong A_4 $ \\ \\
         \end{tabular} 
 \\
\hline

\begin{tabular}{l} 
 \\
           $U_6 = \la a,b,c,x,y \mid a^9,  b^2,c^3, x^2, y^2,
                                                        a^3=c, (ab)^2,  [x,y], (bc)^2, x^c=y, y^c=xy, x^b = y \ra$ \\ \\
           $L_6=\la a,b, c\ra  \cong D_9, \quad B_6=\la b,c \ra \cong S_3,\quad R_6=\la b,c,x,y\ra \cong S_4 $ \\ \\
         \end{tabular} 
 \\
 \hline

\begin{tabular}{l} 
 \\
           $U_7 = \la a,b,c,x,y \mid a^3,  b^2,c^3, x^2, y^2,
                                                        [x,y], (bc)^2, x^c=y, y^c=xy, x^b = y, [a,c], (ab)^2 \ra$ \\ \\
           $L_7=\la a, b, c\ra  \cong \GenDih(C_3\times C_3), \quad B_7=\la b,c \ra \cong S_3,\quad R_7=\la b,c,x,y\ra \cong S_4 $ \\ \\
         \end{tabular} 
 \\
 \hline

\begin{tabular}{l} 
 \\
           $U_8 = \la a,b,c,x,y \mid a^3,  b^2,c^3, x^2, y^2,
                                                        [x,y], (bc)^2, x^c=y, y^c=xy, x^b = y, [a,c], [a,b] \ra$ \\ \\
           $L_8=\la a, b, c\ra = \la c \ra \rtimes \la a,b\ra \cong C_3\rtimes C_6, \quad B_8=\la b,c \ra \cong S_3,\quad R_8=\la b,c,x,y\ra \cong S_4 $ \\ \\
         \end{tabular} 
 \\
 \hline

\begin{tabular}{l} 
 \\
           $U_9 = \la c, d,x,y \mid c^3,d^2,x^2,y^2, (cd)^2, [d,x], [d,y], [x,y] \ra $ \\ \\
           $L_9=\la c,d\ra  \cong S_3, \quad B_9=\la d \ra \cong C_2,\quad R_9=\la d,x,y\ra \cong C_2\times C_2\times C_2 $ \\ \\
         \end{tabular} 
 \\
 \hline

\begin{tabular}{l} 
 \\
           $U_{10} =\la c,d,x \mid c^3,d^2,x^4,  (cd)^2, [x,d] \ra $ \\ \\
           $L_{10}=\la c,d\ra  \cong S_3, \quad B_{10}=\la d \ra \cong C_2,\quad R_{10}=\la d,x\ra \cong C_2\times C_4$ \\ \\
         \end{tabular} 
 \\
 \hline

\begin{tabular}{l} 
 \\
           $U_{11} =\la c,d,x,y \mid c^3,d^2,x^4,y^2, (cd)^2, x^2=d, [x,y] \ra $ \\ \\
           $L_{11}=\la c,d\ra  \cong S_3, \quad B_{11}=\la d \ra \cong C_2,\quad R_{11}=\la d,x,y\ra \cong C_2\times C_4$ \\ \\
         \end{tabular} 
 \\
 \hline

\begin{tabular}{l} 
 \\
           $U_{12} =\la c,d,x \mid c^3, d^2,x^8, (cd)^2, x^4=d \ra $ \\ \\
           $L_{12}=\la c,d\ra  \cong S_3, \quad B_{12}=\la d \ra \cong C_2,\quad R_{12}=\la d,x\ra \cong C_8$ \\ \\
         \end{tabular} 
 \\
 \hline

\begin{tabular}{l} 
 \\
           $U_{13} =\la c,d,x,y \mid c^3, d^2,x^4,y^2, (cd)^2, x^2=d, (xy)^2 \ra $ \\ \\
           $L_{13}=\la c,d\ra  \cong S_3, \quad B_{13}=\la d \ra \cong C_2,\quad R_{13}=\la d,x,y\ra \cong D_4$ \\ \\
         \end{tabular} 
 \\
 \hline

\begin{tabular}{l} 
 \\
           $U_{14} = \la c,d,x,y \mid c^3, d^2,x^4,y^4, (cd)^2, x^2=y^2=[x,y]=d \ra $ \\ \\
           $L_{14}=\la c,d\ra  \cong S_3, \quad B_{14}=\la d \ra \cong C_2,\quad R_{14}=\la d,x,y\ra \cong Q$ \\ \\
         \end{tabular} 
 \\
\hline

\begin{tabular}{l} 
\\
         $U_{15} = \la a,c,d,x,y \mid a^3, c^3, d^2,x^2, y^2,
                                                         (dc)^2, [a,c], [a,d],
                                                       [d,x],[d,y], [x,y],  x^a=y, y^a=xy \ra$
                                                       \\ \\                                                      
           $L_{15}=\la a,c,d\ra \cong C_3 \times S_3, \quad B_{15}=\la a ,d \ra \cong C_6,\quad R_{15}=\la a,d,x,y\ra \cong C_2 \times A_4 $ \\ \\
  \end{tabular} 
  \\
\hline

\begin{tabular}{l} 
\\
         $U_{16} = \la a,c,d,x,y \mid a^3, c^3, d^2,x^4, y^4,
                                                         (dc)^2, [a,c], [a,d],
                                                       x^2=y^2=[x,y] = d,  x^a=y, y^a=xy \ra$
                                                       \\ \\                                                      
           $L_{16}=\la a,c,d\ra \cong C_3 \times S_3, \quad B_{16}=\la a ,d \ra \cong C_6,\quad R_{16}=\la a,d,x,y\ra \cong Q \rtimes C_3 $ \\ \\
  \end{tabular} 
  \\
\hline

\begin{tabular}{l} 
\\
  \begin{tabular}{ll}
         $U_{17} = \la a,b,c,d,x,y \mid$ & \hspace{-3mm} $a^3, b^2, c^3, d^2,x^2, y^2,
                                                        (ba)^2, (dc)^2, [a,c], [a,d], [b,c],[b,d],$\\
                                                      & $  [x,d],[y,d], [x,y],  x^a=y, y^a=xy, x^b=x, y^b=xy \ra$
  \end{tabular}                                                    
                                                       \\ \\
                                                      
           $L_{17}=\la a,b,c,d\ra =  S_3 \times S_3, \quad B_{17}=\la a, b,d \ra \cong C_2 \times S_3,\quad R_{17}=\la a,b,d,x,y\ra \cong C_2 \times S_4 $ \\ \\
  \end{tabular} 
  \\
\hline

\begin{tabular}{l} 
\\
  \begin{tabular}{ll}
         $U_{18} = \la a,b,c,d,x,y \mid$ & \hspace{-3mm} $a^3, b^2, c^3, d^2,x^4, y^4,
                                                        (ba)^2, (dc)^2, [a,c], [a,d], [b,c],[b,d],$\\
                                                      & $  x^2=y^2=[x,y]=d,  x^a=y, y^a=xy, x^b=x^{-1}, y^b=yx \ra$
  \end{tabular}     
 \\ \\
           $L_{18}=\la a,b,c,d\ra =  S_3 \times S_3, \quad B_{18}=\la a, b,d \ra \cong C_2 \times S_3,\quad R_{18}=\la a,b,d,x,y\ra \cong Q \rtimes S_3 $ \\ \\
  \end{tabular} 
  \\
\hline

\hline

\end{tabular}
\end{scriptsize}
\end{center}

As an application of Theorem~\ref{the:main}, we compute a complete list of all connected biregular graphs of valence $\{3,4\}$ on
 at most $350$ vertices (and thus at most $600$ edges)
admitting a locally arc-transitive group of automorphisms $G$ with $G_{uv}^{[1]} = 1$.
This was done by applying the  {\tt LowIndexNormalSubgroups} routine, implemented in {\tt Magma} \cite{magma},
to find, for each of the groups
$U_i$, all the quotients $G$ of $U_i$ of order at most $600 |B_i|$ such that the group $L_i$ and $R_i$ project isomorphically onto
some subgroups $L$ and $R$ of $G$. Once such triples $(G,L,R)$ were obtained, we constructed for each of them 
the so called {\em coset graph} $\Cos(G,L,R)$, the vertex set of which is the disjoint union $G/L \cup G/R$ of coset sets
 with edges of the form $\{Lg,Rg\}$ for $g\in G$ (see, for example, \cite{GLP} for details).
 In this way, we found $220$ pairwise non-isomorphic graphs. A complete list of graphs (in magma code) 
can be accessed at \cite{web}.
In Section~\ref{sec:data}, we present some graph-theoretical parameters of the the $42$ graphs from that list with at most $100$ vertices.

Let us also mention that a triple of finite groups $(L,B,R)$ with $B=L\cap R$ (such as a triple $(L_i,B_i,R_i)$ from Table~1) is often called 
a {\em finite group amalgam} of {\em index} $([L:B],[R:B])$ (that is, of index $(3,4)$ in the case of the amalgams $(L_i,B_i,R_i)$ from Table~1).
Note that the edge-stabiliser $G_{uv}$ in a connected locally $G$-arc-transitive graph is core-free in $G$ (that is, contains no nontrivial subgroups 
that are normal in $G$). This implies that the amalgams $(L,B,R)$ (where $L=G_v$, $B=G_{uv}$ and $R=G_u$)
arising from such pairs $(\Gamma,G)$ have the following property: if $N$ is a subgroup of $G_{uv}$ which is normal both in $G_v$ and $G_u$, then $N$ is trivial. Such amalgams are often called {\em faithful}. Furthermore, the requirement that the edge kernel $G_{uv}^{[1]}$ is trivial translates into the requirement that the intersection $\core_L(B) \cap \core_R(B)$  of the cores of $B$ in $L$ and $R$ is trivial. We say that such amalgams
have a {\em trivial edge kernel}. With this terminology in mind,
Theorem~\ref{the:main} can thus also be viewed as a classification of faithful finite group amalgams of index $\{3,4\}$ with trivial edge kernel.
(We refer the reader to \cite{amalgams} for further details on the relationship between amalgams and locally arc-transitive graphs.)

\begin{corollary}
If $(L,B,R)$ is a finite faithful amalgam of index $\{3,4\}$ with trivial edge kernel, then it is isomorphic to one of the amalgams $(L_i,B_i,R_i)$
given in Table~1.
\end{corollary}

We prove Theorem~\ref{the:main} in Section~\ref{sec:proof}. In Section~\ref{sec:large}, we present a construction (and characterisation)
of locally arc-transitive graphs with arbitrary large kernel. Finally, in Section~\ref{sec:data}, a list of all connected graphs of valence $\{3,4\}$
of order up to $100$ that admit a locally arc-transitive action of a group with a trivial edge kernel, is given and several graph theoretical parameters
of these graphs are computed.

\section{Proof of Theorem~\ref{the:main}}
\label{sec:proof}

 Our approach to the proof of Theorem~\ref{the:main} can be summarised as follows: Since $G_{uv}^{[1]} = 1$,
it follows that $G_{uv}$ acts faithfully on $\Gamma(uv)$ and is thus a subgruop of $S_3\times C_2$.
 This gives only finitely possibilities for $G_{uv}$ as an abstract group.
Furthermore, since $G_{uv}$ is embedded in $G_u$ and $G_v$ as a subgroups of index $4$ and $3$, 
respectively, there is only finitely many  possible embeddings of
$G_{uv}$ into $G_u$ and $G_v$. For each possible pair of embeddings, we shall write the groups 
$G_u$ and $G_v$ as finitely presented groups with the common generators
generating the group $G_{uv}$. Finally, since the group $G$ is generated by $G_u$ and $G_v$, 
this will allow us to conclude that $G$ is a quotient of the group
generated by the union of the generators of $G_u$ and $G_v$ subject to the union of relators of $G_u$ and $G_v$, respectively.
We refer the reader to \cite{pot}  for the details of this procedure.
\medskip

To simplify notation, let $K_v = G_v^{[1]} $ and $K_u=G_u^{[1]}$.
Since $G_{uv}^{[1]}=1$, it follows that  the mapping  which maps an element $g\in G_{uv}$ to
the permutation $\bar{g}$ induced by the action of $g$ on $\Gamma(uv)$ is an isomorphism of groups $G_{uv}$ and $G_{uv}^{\Gamma(uv)}$.
Being a kernel of group epimorphisms, the groups $K_u$ and $K_v$ are isomorphic to some normal subgroups of $\Sym(\Gamma(u)\setminus\{v\}) \cong C_2$
and $\Sym(\Gamma(v)\setminus\{u\}) \cong S_3$, respectively.
This shows that  $K_u$ is either trivial or isomorphic to $C_2$, while
$K_v$ is either trivial, isomorphic to $C_3$, or to $S_3$.
We will split the proof into two cases, depending on whether $K_u$ is trivial or not.
\medskip

{\bf Case A}
 Suppose that $K_u=1$. Then, of course, $G_u$ acts faithfully and transitively on $\Gamma(u)$ and $G_u \cong G_u^{\Gamma(u)} \le S_4$.
 The group $G_v$ is therefore isomorphic to one of the groups $C_4$, $C_2\times C_2$,  $D_4$, $A_4$ or $S_4$.
\medskip

{\bf Case A.1}
Let us first deal with the case where $G_u$ is isomorphic to $C_4$ or $C_2\times C_2$.
Then $G_{uv} = 1$ and thus $K_v = 1$, implying that $G_v \cong G_v^{\Gamma(v)}$ is a regular subgroup of $S_3$; in particular, $G_v \cong C_3$.
Since $G$ is generated by $G_u$ and $G_v$, this implies that $G$ is a quotient of the free product $U_0\cong C_3*C_4$ or of the free product $U_1\cong C_3*(C_2\times C_2)$, with $G_v \cong L_i$, $G_u\cong R_i$ and $G_{uv} \cong B_i$ for $i=0$ or $1$.
\medskip

{\bf Case A.2}
Suppose now $G_u\cong  D_4$. Let $b\in G_u$ be the involution fixing $v$ and let $c\in G_u$ be an element that cyclically permutes the neighbours of $u$.
Note that $G_u=\la b,c \mid b^2, c^4, (bc)^2\ra$ and $G_{uv} = \la b \ra$. Furthermore, since $|G_v| = 3 |G_{uv}| = 6$, it follows that $G_v \cong C_6$ or
$S_3$. In both cases, let $a$ be a generator of the unique subgroup of order $3$ in $G_v$. Then $G_v = \la a,b\ra$ and we see that $[a,b] = 1$ if $G_v \cong C_6$ and
$(ab)^2=1$ if $G_v \cong S_3$. Since $G$ is generated by $G_v$ and $G_u$, it follows that $G$ is a quotient of the group 
$U_2$ (if $G_v\cong C_6$) or of $U_3$ if ($G_v\cong S_3$).

{\bf Case A.3}
Suppose finally that $G_u\cong A_4$ or $S_4$. Let $x$ and $y$ be the generators of the regular normal subgroup of $G_u$, isomorphic to $C_2\times C_2$.
Observe that $G_u = \la x, y\ra \rtimes G_{uv}$ where $G_{uv}  \cong C_3$ (if $G_u \cong A_4$) or $G_{uv}  \cong S_3$ (if $G_u \cong S_4$). 
Note also that the action of $G_{uv}$ on $\Gamma(v) \setminus \{u\}$ is equivalent to the action of $G_{uv}$ by conjugation on the nontrivial elements of $\la x,y\ra$.
Let $c$ be the element of order $3$ in $G_{uv}$ that cyclically permutes the elements $x,y,xy$ in that order, and
if $G_u\cong S_4$, let $b$ be the involution of $G_{uv}$ for which $y=x^b$.
The generators $x,y,c$ (and possibly $b$) of $G_v$ then satisfy the relations:
$$
 x^2=y^2=[x,y]=c^3=b^2=1,\> x^c=y, y^c = xy, x^b =y, (cb)^2 = 1.
$$
Note that since $c$ is of order $3$, it acts trivially on $\Gamma(v)\setminus \{u\}$; that is $c\in K_v$. Moreover, since $\la c \ra$ is characteristic in $G_{uv}$,
it follows that $\la c \ra$ is normal in $G_v$.

If $G_u \cong A_4$ (and thus $G_{uv} = \la c \ra = K_v$), then $|G_v| = 3 |G_{uv}| = 9$ and therefore $G_v \cong C_9$ or $C_3 \times C_3$.
In the former case, we see that $G_v$ is generated by some element, say $a$, of order $9$ satisfying $a^3=c$. Since $G$ is generated by $G_v$ and $G_u$,
it must be a quotient of the group of the group $U_4$ in Table 1.
On the other hand, if $G_v \cong C_3\times C_3$, then it is generated by $c$ and some element $a$ of order $3$ satisfying
$[a,c] = 1$, and $G$ is a quotient of the group $U_5$ in Table 1.

Suppose now that $G_u \cong S_4$ (and thus $G_{uv} = \la b,c \ra$). Then either $K_v = G_{uv}$, implying that $G_v^{\Gamma(v)} \cong C_3$, or $K_v = \la c \ra$ and thus
$G_v^{\Gamma(v)} \cong S_3$. As observed above, in both cases $\la c \ra$ is normal in $G_v$. Consider the Sylow $3$-subgroup $P$ of $G_v$.
Since $|G_v| = 3 |G_{uv}| = 18$, it follows that $P$ is normal of index $2$ in $G_v$ and is isomorphic either to $C_3\times C_3$ or to $C_9$. Since
$G_v$ contains an involution $b$, it follows that $G_v = P \rtimes \la b\ra$.

 If $P =  C_9$, then let $a$ be its generator such that $a^3 = c$.
 Since $G_v$ contains $G_{uv} \cong S_3$ and is thus  nonabelian,
it follows that $G_v = \la a,b\ra \cong D_9$ (note that $D_9$ is the only nonabelian group of order $18$ containing a cyclic subgroup of order $9$). 
In particular, $G_v$ is generated by $a$ and $b$, subject to relations $a^9=b^2 = (ab) =1$.
The group $G$ is then a quotient of the group $U_6$ in Table~1.

 If $P =  C_3 \times C_3$, then consider the action by conjugation of $\la b \ra \cong C_2$ on the three complements of $\la c \ra$ in $P$.
 This action has at least one fixed point, say $\la a \ra \cong C_3$. The two possibilities for the action of $b$ by conjugation on $\la a \ra$ give rise to
 two possibilities for the stabiliser $G_v$:
 \begin{eqnarray*}
  G_v & = &\la a,b,c \mid a^3,c^3,b^2, [a,c], (bc)^2, [a,b] \ra = \la c \ra \rtimes \la a,b\ra \cong C_3 \rtimes C_6;\\
  G_v & = & \la a,b,c \mid a^3,c^3,b^2, [a,c], (bc)^2, (ab)^2 \ra \cong \GenDih(C_3\times C_3).
 \end{eqnarray*}
 The group $G$ is thus a quotient of the group $U_7$ or the group $U_8$ in Table~1.
 \medskip

{\bf Case B}  Suppose now that $K_u \not = 1$, and thus $K_u \cong C_2$.
Let $d$ be the generator of $K_u$.  Since $K_u$ acts transitively on the set $\Gamma(v)\setminus \{u\}$,
and since $K_v$ is equal to the point-stabiliser of the action of the group $G_{uv}$ on $\Gamma(v)\setminus \{u\}$, it follows that 
$$
 G_{uv}= \la K_v, K_u\ra = K_v \times K_u = K_v \times \la d \ra \cong K_v \times C_2.
$$
Since $G_{uv}$ contains $d$, it follows that $G_v^{\Gamma(v)} \cong S_3$.
Moreover, in view of the isomorphism $G_{uv} \cong G_{uv}^{\Gamma(uv) \setminus \{u,v\}}$,
it follows that $K_v$ is isomorphic to a normal subgroup of $\Sym(\Gamma(u) \setminus \{v\}) \cong S_3$.
 We thus need to consider the cases $K_v = 1$, $K_v \cong C_3$ and $K_v \cong S_3$.
\medskip

{\bf Case B.1} Suppose first that $K_v = 1$.
Then $G_v \cong G_v^{\Gamma(v)}$ and hence $G_v = \la c,d \ra \cong S_3$, where $c$ is an arbitrary element of $G_v$ or order $3$ (and thus satisfying the relation $(cd)^2=1$).
On the other hand, $|G_u| = 4 |G_{uv}| = 8$, implying that $G_u$ is one of the five groups of order $8$ (the three abelian groups of order $8$, the dihedral group $D_4$
or the quaternion group $Q$). 
Note that apart from $C_4\times C_2$, in the remaining four groups of order $8$,
the central involutions are conjugate under the automorphism group of the group, implying that there is essentially a unique way 
how to embed $d$ into any of these groups. On the other hand, there are two types of involutions in $C_4\times C_2$, those that
are squares of elements of order $4$ and those that are not. Here the element $d$ can be embedded in two essentially distinct ways.
 This therefore gives rise to six possible stabilisers $G_u$, listed below, each generating together with $G_v = \la c, d\ra$
a quotient of one of the groups $U_9, \ldots, U_{14}$ in Table~1.
 \begin{eqnarray*}
  G_u & = &\la d,x,y \mid d^2,x^2,y^2, [d,x], [d,y], [x,y] \ra  \cong C_2 \times C_2 \times C_2;\\
  G_u & = &\la d,x \mid d^2,x^4,  [x,d] \ra  \cong C_4\times C_2;\\
  G_u & = &\la d,x,y \mid d^2,x^4,y^2, x^2=d, [x,y] \ra  \cong C_4\times C_2;\\
  G_u & = &\la d,x \mid d^2,x^8, x^4=d \ra  \cong C_8;\\
  G_u & = &\la d,x,y \mid d^2,x^4,y^2, x^2=d, (xy)^2 \ra  \cong D_4;\\
  G_u & = &\la d,x,y \mid d^2,x^4,y^4, x^2=y^2=[x,y]=d \ra  \cong Q.\\
 \end{eqnarray*}

{\bf Case B.2} Suppose now that $K_v \cong C_3$. Then $K_v = \la a \ra$ for some element $a$ of order $3$ which cyclically permutes the
elements of $\Gamma(u) \setminus \{v\}$. Then 
$$
 G_{uv} = K_v \times K_u = \la a,d\ra \cong C_6.
$$
 Moreover, since $G_{uv}$ is transitive both
on $\Gamma(u)\setminus \{v\}$ as well as on $\Gamma(v)\setminus \{u\}$, it follows that $G_u^{\Gamma(u)}$ and $G_v^{\Gamma(v)}$ are doubly transitive groups
and hence $\Gamma$ is locally $(G,2)$-arc-transitive. In particular, $G_v^{\Gamma(u)} \cong S_3$ 
and since $G_{uv}^{\Gamma(u)} \cong C_3$, also $G_u^{\Gamma(u)} \cong A_4$.

Let us now determine the structure of the group $G_v$. Observe first that $|G_v| = 3 |G_{uv}] = 18$
and let $P$ be a Sylow $3$-subgroup of $G_v$. Since $[G_v:P] = 2$, we see that $P$ is normal in $G_v$ and therefore  $a\in P$.
Moreover, since $d\in G_v \setminus P$, the group $G_v$ splits over $P$ into a semidirect product $P \rtimes \la d \ra$. Now consider the
action of $\la d \ra$ upon $P$ by conjugation. If this action is trivial, then $G_v$ is abelian and $K_u = \la d \ra$ 
is normal in both $G_v$ and $G_u$.
But then $K_u$ is normal in $G$, implying that $K_u$ acts trivially on the set of edges of the graph, which is clearly a contradiction.

If $P \cong C_9$, then $\Aut(P)$ contains a unique involution, namely the one inverting the elements of $P$. Hence
$G_v  \cong D_9$. But since $D_9$ contains no subgroup isomorphic to $C_6\cong G_{uv}$, this cannot occur in this case.

Therefore $P \cong C_3 \times C_3$. Now, similarly as in the last paragraph of Case A.3,
consider the action of $\la d \ra$ on the set of the four subgroups of $P$ of order $3$ by conjugation.
Since $\la d \ra$ has order $2$ and already fixes one such group (namely the group $\la a \ra$), it must fix at least one other; let $c$ be its generator.
Since $d$ does not centralise $P$, but centralises $a$, it follows that
$[c,d] \not = 1$, and thus $c^d = c^{-1}$ (or equivalently $(cd)^2=1$).
This shows that 
\begin{equation}
\label{Gv2}
G_v = \la a,c,d \ra \> \hbox{ where } \> c^3= a^3 = [a,c] =1, d^2=[d,a] = (cd)^2=1. 
\end{equation}
Note that $G_v = \la a \ra \times \la c,d\ra \cong C_3 \times S_3$.

Let us now consider the structure of $G_u$. Let $\pi \colon G_u \to G_u^{\Gamma(u)}$ be the epimorphism that maps each element $g\in G_u$
to the permutation induced by $g$ on $\Gamma(u)$.
Note that the kernel of $\pi$ is $K_u = \la d \ra$ and in particular that $\la d \ra$ is normal (and therefore central) in $G_u$.
Recall that $G_u^{\Gamma(u)} \cong A_4$,
 let $V$ be the regular normal subgroup of $G_u^{\Gamma(u)}$, isomorphic to
the Klein group $C_2\times C_2$, and let $P=\pi^{-1}(V)$. Then $P$ is a normal Sylow $2$-subgroup of $G_u$, implying that $d\in P$ and
$G_u = P \rtimes \la a \ra$. Since $\la a \ra$ is normal in $G_v$, it is not normal in $G_u$ (since otherwise it would act trivially 
on the edge set of $\Gamma$), implying that the action of $\la a \ra$ upon $P$ by conjugation
is nontrivial. Note that the only groups of order $8$ that admit an automorphism of order $3$ are the elementary abelian group $C_2^3$ and the quaternion group $Q$. This leaves us with two possibilities: $P\cong C_2^3$ and $P \cong Q$.
\medskip

{\bf B.2.1} Suppose that $P\cong C_2^3$. Then the fact that $a$ centralises $d$ implies that $\la a \ra$ fixes one of the four complements of $\la d\ra$ in $P$.
Clearly we may choose the generators $x$ and $y$ of that complement in such a way that  $x^a = y$ and $y^a=xy$.
 We have thus shown that $G_u$ is generated by elements $x,y,d,a$, which, in addition to relations in (\ref{Gv2}), satisfy also the following:
  \begin{equation}
   x^2=y^2=[x,d] = [y,d] = [x,y] = 1, x^a = y,  y^a = xy. 
\end{equation}
The group $G=\la G_u, G_v\ra$ must thus be a quotient of the group $U_{15}$ in Table~1. Observe also that $G_u \cong \la d \ra \times \la x,y,a\ra \cong C_2 \times A_4$.
\medskip

{\bf B.2.2} Suppose now that $P$ is isomorphic to the quaternion group $Q$.
Since the centre of $Q$ is of order $2$ and since $d$ is central in $P$, it follows that
$Z(P) = \la d \ra$. Furthermore, since $a$ centralises $d$, the partition $\{ \{g,gd\} : g\in P \setminus \{1,d\}\}$ of the set $P \setminus \{1,d\}$ into
the nontrivial cosets of $\la d \ra$ is $\la a \ra$-invariant. Hence the action of $\la a\ra$ on $P$ by conjugation gives rise
to two orbits on $\{ \{g,gd\} : g\in P \setminus \{1,d\}\}$, each intersecting each pair $\{g,gd\}$ in a unique point.
Moreover, since $P/\la d \ra \cong C_2\times C_2$, if $\{x,xd\}$ and $\{y,yd\}$ are distinct nontrivial cosets, then $\{xy,xyd\}$ is the third one.
Now choose $x_0 \in P \setminus \{1,d\}$ and let $y_0 = x_0^a$.
Then $y_0^a \in \{x_0y_0, x_0y_0d\}$.
If $y_0^a = x_0y_0$, then let $x=x_0$ and $y=y_0$;
otherwise, let $x=x_0d$ and $y=y_0d$. Observe that in both cases it follows that $x^a = y$ and $y^b = xy$.
We have thus shown that
 $G_u$ is generated by elements $x,y,d,a$, which, in addition to relations in (\ref{Gv2}), satisfy also the following:
  \begin{equation}
   x^4=y^4=1, x^2=y^2=[x,y] = d, x^a = y,  y^a = xy. 
\end{equation}
The group $G=\la G_u, G_v\ra$ is therefore a quotient of the group $U_{16}$ in Table~1.
\medskip

{\bf Case B.3} Suppose finally that $K_v \cong S_3$. Then $K_v$ is generated by some elements $a,b$ satisfying $a^3=b^2=(ab)^2 = 1$ and thus
\begin{equation}
\label{Guv1.1}
G_{uv} = K_v \times K_u = \la a,b,d\ra \cong S_3 \times C_2.
\end{equation}

Let us first determine the structure of the vertex-stabiliser $G_v$.
Since the centre and the outer automorphism group of $K_v\cong S_3$ are both trivial,
it is not difficult to see that $K_v$ must be a direct factor in every group in which it is normal (see for example  \cite[13.5.8 and Exercise 13.5]{rob}).
In particular, $G_v =K_v \times H$ where $H \cong H^{\Gamma(v)} = G_v^{\Gamma(v)} \cong S_3$.
Now consider the group $H_u = H \cap G_u = H\cap G_{uv}$. 
Since $H$ is normal in $G_v$, it follows that $H_u$ is normal in $G_v \cap G_u = G_{uv}$.
However, since $H\cong H^{\Gamma(v)} \cong S_3$, we see that $H_u \cong S_2$, implying that $H_u$ is central in $G_{uv}$.
However, the centre of $G_{uv}=\la a,b,d\ra$ is the group $\la d \ra$, showing that
$H_u = \la d \ra = K_u$.

If we let $c$ be an element of order $3$ in $H\cong S_3$, then $H = \la c,d\ra$, with $(dc)^2 =1$.
Note that $c$ commutes with $a$ and $b$ (since $c\in H$ and $G_v =K_v \times H$). This implies that
$G_v$ is generated by elements $a,b,c,d$, subject to relations
\begin{equation}
\label{eq:Gv1}
 a^3=b^2=c^3=d^2=(ba)^2=(dc)^2 = [a,c] = [a,d] = [b,c] = [b,d] =1.
 \end{equation}
 In particular, $G_v = \la a,b \ra \times \la c, d\ra \cong S_3 \times S_3$ is isomorphic to the group $L_4$ in Table~1.

We will now determine the structure of the group $G_u$. Observe first that $|G_u| = 4 |G_{uv}| = 48$. Since $|K_u|=2$, it follows that $|G_u^{\Gamma(u)}| = 24$ and
hence $G_u^{\Gamma(u)} \cong S_4$. Note also that since $\la d \ra = K_u$ is a normal subgroup of $G_u$ of order $2$, the element $d$ is central in $G_u$.
As in Case B.2.1, let $\pi \colon G_u \to G_u^{\Gamma(u)}$ denote the group epimorphism that maps an element $g\in G_u$ to
the permutation induced by the action of $g$ on $\Gamma(u)$. 

Further, let $V$ be the regular normal subgroup of $G_u^{\Gamma(u)}$,
isomorphic to the Klein group $C_2\times C_2$, and let $P = \pi^{-1}(V)$ be its preimage in $G_u$.
Then $P$ is  normal in $G_u$ and has order $8$, Moreover,
$$ 
 P / \la d \ra \> \cong \> V \> \hbox{ and } \>  G_u/ P \> \cong \> G_u^{\Gamma(u)} / V \> \cong\> S_3.
$$
Since $V$ is transitive on $\Gamma(u)$, so is $P$, implying that $G_u = PG_{uv}$. However, $G_{uv} = K_u K_v$ and $K_u\le P$,
showing that $G_u = P K_v$.  Since $|G_u| = 48 = 8\cdot 6 = |P|\, |K_v|$, it follows that $P\cap K_v = 1$ and thus
\begin{equation}
 G_u = P \rtimes K_v.
\end{equation}

Let us now consider the action of $K_v = \la a,b\ra \cong S_3$ on $P$ by conjugation. 
If the kernel of this action is nontrivial, then it contains $a$ (since $\la a \ra$ is the unique minimal
normal subgroup of $K_v$). In this case, $G_u$ contains a subgroup isomorphic to $P\times C_3$,
implying that $G_u^{\Gamma(u)}$ contains a subgroup isomorphic to $(P\times C_3) /\la d \ra \cong V \times C_3$.
This contradicts the fact that $G_u^{\Gamma(u)}$ is isomorphic to $S_4$, and thus
shows that $K_v \cong S_3$ embeds into $\Aut(P)$. 
Out of five groups of order $8$ (the three abelian ones,
the dihedral group $D_4$ and the quaternion group) only the elementary abelian group of order $8$ and the quaternion group admit $S_3$ as a group of automorphisms.
\medskip

{\bf Case B.3.1} Suppose $P$ is elementary abelian. Then we proceed similarly as in Case B.2.1.
Since $K_v$ centralises $d$, it permutes (via conjugation) the set of $4$ complements of $\la d \ra$ in $P$.
Since every action of $S_3$ on $4$ points has at least one fixed point, this implies that $K_v$ normalises at least one complement
of $\la d \ra$ in $P$, say $W$.
Since $K_v$ acts trivially on $\la d \ra$, it must acts faithfully as a group of automorphisms of $W$. But since $\Aut(W) = \GL(2,3) \cong S_3$, the action of $K_v$
on $W$ is uniquely determined. In particular, we may choose generators $x$ and $y$ of $W$ in such a way that $x^a = y$, $y^a=xy$, $x^b = x$ and $y^b = xy$.

 We have thus shown that $G_u$ is generated by elements $x,y,d,a,b$, which, in addition to relations in (\ref{eq:Gv1}), satisfy also the following:
  \begin{equation}
\label{eq:Gu1}
    x^2=y^2=[x,d] = [y,d] = [x,y] = 1, x^a = y,  y^a = xy,  x^b = x,  y^b = xy. 
\end{equation}
Note that $G_u = \la d,x,y\ra \rtimes \la a,b\ra \cong (C_2^3)\rtimes S_3$, as well as $G_u = \la d \ra \times \la x,y,a,b \ra \cong C_2 \times S_4$. In particular,
$G_u$ is isomorphic to the group $R_{17}$ in Table 1.
Finally, since $G$ is generated by $G_v$ and $G_u$, formulas (\ref{eq:Gv1}) and (\ref{eq:Gu1}) show that $G$ is a quotient of the group $U_{17}$ in Table 1.
\medskip

{\bf Case B.3.2} Suppose now that $P$ is the quaternion group. This case in analogous to Case B.2.2.
Since the centre of $P$ is of order $2$ and since $d$ is central in $P$, it follows that
$Z(P) = \la d \ra$. Furthermore, since $K_v$ centralises $d$, the partition $\{ \{g,gd\} : g\in P \setminus \{1,d\}\}$ of the set $P \setminus \{1,d\}$ into
the nontrivial cosets of $\la d \ra$ is $K_v$-invariant.
Hence the action of $K_v=\la a\ra$ on $P$ by conjugation gives rise two orbits on $\{ \{g,gd\} : g\in P \setminus \{1,d\}\}$, 
each intersecting each pair $\{g,gd\}$ in a unique point.
Moreover, since $P/\la d \ra \cong C_2\times C_2$, if $\{x,xd\}$ and $\{y,yd\}$ are distinct nontrivial cosets, then $\{xy,xyd\}$ is the third one.
Clearly $b$ fixes at least one coset setwise, say $\{x_0, x_0d\} \subseteq P \setminus \{1,d\}$.
Let $y_0 = x_0^a$ and note that $y_0^a \in \{x_0y_0, x_0y_0d\}$.
If $y_0^a = x_0y_0$, then let $x=x_0$ and $y=y_0$;
otherwise, let $x=x_0d$ and $y=y_0d$. In both cases we see that $x^a = y$.
Moreover, in the former case, we clearly have also $y^a = xy$. 
However, even in the latter case, we see that $y^a = (y_0d)^a = y_0^ad = (x_0y_0d)d = (x_0d)(y_0d) = xy$.
Hence in both cases we see that $y^a = xy$.
On the other hand, since we have chosen $x_0$ in such a way that $b$ fixes $\{x_0, x_0d\}$ setwise,
we have either $x^b = x$ or $x^b=xd$. If $x^b = x$, then 
$y^b = x^{ab} = (x^b)^{a^{-1}} = xy$,
and since $b$ is an involution, also $(xy)^b = y$.
But then 
$y = (xy)^b = (dyx)^b = d^by^bx^b=dxyx = yx^2=yd$.
This contradiction shows that $x^b = xd = x^{-1}$, and consequently, $y^b = x^{ab} = (x^b)^{a^{-1}} = (xy)^{-1} = xyd=yx$.
Since any two generators $x,y$ of the quaternion group satisfy the relations $x^4=y^4 = 1$, $x^2=y^2=[x,y] = d$ (where $d$ is the central involution), it follows that
$G_u$ is generated by elements $x,y,d,a,b$ which, in addition to (\ref{eq:Gv1}) satisfy also the relations:
 \begin{equation}
  \label{eq:Gu2}
    x^4=y^4=1, x^2=y^2=[x,y] = d,  x^a = y,  x^b = x^{-1},  y^a = xy,   y^b = yx. 
\end{equation}
Since $G$ is generated by $G_v$ and $G_u$, formulas (\ref{eq:Gv1}) and (\ref{eq:Gu2}) show that $G$ 
is a quotient of the group $U_{18}$ in Table 1. This completes the proof of Theorem~\ref{the:main}.

\section{Locally arc-transitive group actions with large edge kernel}
\label{sec:large}

In this section we first show that the edge kernel in a locally arc-transitive graph of valence $\{3,4\}$
can be arbitrary large. This will be proved by means of the so called {\em subdivided doubles},
a construction that was introduced in \cite{PW} and that can be described as follows.

Let $\Lambda$ be a  $G$-arc-transitive $k$-valent graph. The {\em subdivision of $\Lambda$},
denoted by $\SS\Lambda$,
 is the bipartite graph of valence $\{2,k\}$ with vertex set $V(\Lambda) \cup E(\Lambda)$  in which 
each $e\in E(\Gamma)$ is adjacent to the endpoints of $e$ in $\Lambda$. The {\em subdivided double}
of $\Lambda$ is then the graph obtained from $\SS\Lambda$ by blowing up each original vertex $v\in V(\Lambda)$
to a pair of vertices, each being adjacent to the neighbours of $v$ in $\SS\Lambda$. More precisely, the vertex set of
$\DD_2$ can be defined as the disjoint union $(\ZZ_2\times V(\Lambda)) \cup E(\Lambda)$ with edges of the form
$(i,v)e$ for any $i\in \ZZ_2$, $v\in V(\Lambda)$  and $e\in E(\Lambda)$ such that $v$ is an endpoint of $e$. Note that for any
$v\in V(\Lambda)$, the permutation that interchanges the vertices $(0,v)$ and $(1,v)$ in $\DD_2\Lambda$ and
fixes all other vertices of $\DD_2\Lambda$ is an automorphism of $\DD_2\Lambda$. The group generated by all such
automorphisms is elementary abelian of order $2^{|V(\Lambda)|}$. Together with the group induced by the obvious 
action of $G$ on $\DD_2\Lambda$, this group generates a group which acts locally arc-transitively on $\DD_2\Lambda$
and has a large edge kernel. This can be summarised as follows (the proof can be found in \cite[Lemma 4.2]{PW}).

\begin{lemma}
\label{lem:largekernel}
 If $\Lambda$ is arc-transitive graph of valence $k$ and with $n$ vertices, then $\DD_2\Lambda$ is a locally $G$-arc-transitive graph of valence
 $\{2,k\}$ for some group $G$ with the edge kernel $G_{uv}^{[1]}$ of order divisible by $2^{n-2}$.
\end{lemma}

%
%

Following \cite[Definition~4.1]{PW} we shall call a graph {\em unworthy} if two of its vertices share the same neighbourhood.
The following lemma is a converse of Lemma~\ref{lem:largekernel} in the context of locally $G$-arc-transitive graphs of valence $\{3,4\}$.

\begin{lemma}
\label{lem:unworthy}
Let $\Gamma$ be a connected unworthy locally $G$-arc-transitive graph of valence $\{3,4\}$. Then
$\Gamma$ is isomorphic either to the complete bipartite graph $K_{3,4}$ or to the subdivided double $\DD_2 \Lambda$
for some connected cubic $G$-arc-transitive graph $\Lambda$.
\end{lemma}

\begin{proof}
For every vertex $u$ of $\Gamma$,
let $B(u)$ be the set of all vertices of $\Gamma$ that have the same neighbourhood as $u$. Since $\Gamma$ is unworthy,
we see that for some vertex $u$, the size of $B(u)$ is at least $2$; let $U$ be the part of the bipartition of $\Gamma$ containing $u$
and let $W=V(\Gamma) \setminus U$.

Since $\Gamma$ is locally $G$-arc-transitive, the group $G$ acts transitively on $U$ as wee as on $W$
and the set $\P=\{B(v) : v \in U\}$ is a $G$-invariant partition of $U$.
In particular, every $w\in W$ is adjacent to a fixed number of blocks in $\P$.
Observe also that for any given $B\in \P$ and $w\in W$, the vertex $w$ is adjacent 
either to none or to all of the vertices in $B$. In particular, the size of $B$ is a divisor of the valence of $w$. Since $|B| \ge 2$,
this shows that  either the valence of $w$ is $4$ and $|B| = 2$, or the valence of $w$ equals the size of $B$. In the latter case
each vertex of $W$ is adjacent to precisely one block in $\P$, which, together with connectivity of $\Gamma$, implies that
$\Gamma\cong K_{3,4}$. 

We may thus assume that $|B| =2$ and that the valence of $w$ is $4$ and that no two vertices of valence $4$ share
the same neighbourhood.
 Now consider the bipartite graph $\Gamma_\P$ the vertex set
of which is $\P \cup W$ and with a vertex $w$ adjacent to a vertex $B\in \P$ if and only if $w$ is adjacent in $\Gamma$ to the vertices
in $B$. Since each $w \in W$ is adjacent to precisely two blocks in $\P$ and each block in $\P$ is adjacent to three vertices in $W$,
this graph is biregular of valence $\{2,3\}$. Observe also that the action of $G$ on $V(\Gamma)$ induces a faithful locally arc-transitive
action of $G$ on $V(\Gamma_\P)$ as a group of automorphisms of $\Lambda$. 

We shall now prove that $\Gamma_\P$ contains no cycles of length $4$.
Indeed, if $B,C\in \P$ and $w,v\in W$ are such that
$BwCv$ is a cycle of length $4$ in $\Gamma_\P$, then $w$ and $v$ are both adjacent to all four vertices
of $B\cup C$ in $\Gamma$. But then $w$ and $v$ share the same neighbourhood, which contradicts our assumption that
no two vertices of degree $4$ have that property. This shows that $\Gamma_\P$ contains no cycles of length $4$.


It is now clear that if one suppresses all the vertices of $\Gamma_\P$ of valence $2$ (that is, deletes every vertex $v$ of degree $2$ and
replaces the $2$-path having $v$ as a middle vertex by a single edge), a cubic $G$-arc-transitive graph $\Lambda$ is obtained.
Moreover, it follows easily from the definition of the operator $\DD_2$ and the way we obtained $\Gamma_\P$ that
$\Gamma \cong \DD_2\Lambda$.
\end{proof}

\section{The list of graphs}
\label{sec:data}

As described in Section~\ref{sec:intro}, one can use Theorem~\ref{the:main} to compile a complete list of 
graphs of valence $\{3,4\}$ and bounded order admitting a locally arc-transitive group of automorphisms 
with trivial edge kernel. We have carried out this computation for graph of order at most $350$ and thus obtained
$220$ pairwise non-isomorphic graphs, which are accessible in magma code at \cite{web}.

In Table~2 below, some graph-theoretical parameters for the $42$ graphs on up to $100$ vertices are computed.
Each line in the table corresponds to one of the $42$ graphs. The first item in each line is the ID of the graph and has form $[n,i]$,
where $n$ is the number of the vertices of the graph and $i$ is an index of that graph within the family of all the graphs on $n$
vertices in the table. (The graph with ID $[n,i]$ is stored at \cite{web} in a magma code under the name ${\rm LAT34}[n,i]$.)
Next three items in each line correspond to the girth (the length of a shortest cycle), diameter, and worthiness of the graph.
Then a few parameters pertaining the full automorphism group $A$ of the graph are given. First, for a pair of
adjacent vertices $v$ and $u$ of valency $3$ and $4$, respectively, the permutation groups $A_v^{\Gamma(v)}$ and
$A_u^{\Gamma(u)}$ are computed. Further, the parameters $s_v$ and $s_u$, corresponding to the largest integers $s$
such that $A_v$ (resp.\ $A_u$) is transitive on the $s$-arcs starting at vertex $v$ (resp.\ $u$), are given. The next two items
are the orders of the edge stabiliser $A_{uv}$ and the edge kernel $A_{uv}^{[1]}$ in $A$.
Note that, since the full automorphism group $A$ might be
larger than any of the groups $G$ arising as quotients of the groups $U_i$, the edge kernel $A_{uv}^{[1]}$ can be non-trivial.
The last item in a line is named ``comments'' and gives extra information, such as, whether the graph is isomorphic to
some well-known graph or a subdivided double of a graph appearing in the census of cubic arc-transitive graphs \cite{ConDob}
(here $Q_3$ stands for the graph of the cube,
$\Pet$ stands for the Petersen graph, and the symbols F014,F016, F018, F020A, F020B, F024, F026, F028 
stand for the cubic graphs described in \cite{ConDob}); further comments are provided below.

Let us mention at this point two interesting connections of the topic of this paper with two, at first glance unrelated fields of mathematics.

The first one is the theory of coset geometries which stems from the ingenious work of Tits \cite{tits}.
A {\em coset geometry of rank $2$} is simply a triple of groups $(G:G_0,G_1)$ with $G_i \le G$ for $i\in \{0,1\}$.
To such a coset geometry, one can associate an {\em incidence geometry of rank $2$} with points and lines being the cosets of $G_0$ and $G_1$
in $G$, respectively, and a point $G_0a$ being incident with a line $G_1b$ if and only if $G_0a\cap G_1b \not = \emptyset$.
Note that the incidence graph of this geometry is precisely the coset graph $\Cos(G,G_0,G_1)$. Coset geometries of rank $2$
and the corresponding incidence structures have received a considerable attention in the last few years (see \cite{M,SL}, for example).
The results of this paper can be easily interpreted in the language of that particular field of mathematics.

Let us finally mention an interesting relationship of the graphs in Table~2 with flag-transitive  configurations.
A {\em configuration of type} $(v_r;b_k)$ is determined by the set $\P$ of $v$ points and a set $\B$ of $b$ subsets  of $\P$ of size $k$
such that every point in $\P$ lies in precisely $r$ sets of $\B$ and such that every pair of points lies in at most one set in $\B$.
A pair $(x,B)$ with $x\in \P$, $B\in \B$ and $x\in B$ is called a {\em flag}. The configuration is {\em flag-transitive} if the automorphism
group of the configuration acts transitively on the flags. The incidence graph of the configuration (also called the {\em Levi graph})
is the graph with vertex set $\P \cup \B$ and edges of the form $xB$ where $(x,B)$ is a flag of the configuration. It is easy to see that
the incidence graph of a flag-transitive configuration is locally arc-transitive and of girth at least $6$; conversely, every locally arc-transitive graph
of girth at least $6$ is an incidence graph of some configuration
 (we refer to \cite{baum} for general informaiton about configurations and to \cite{TomoDragan} for relationship between flag-transitive configurations and locally arc-transitive graphs). 

In the above sense, those graphs in Table~2 that have girth at least $6$ correspond to some flag-transitive configurations
of type $(v_4,b_3)$ (with those of girth at least $8$ corresponding to triangle-free configurations; see \cite{BGPZ}). Configurations of type
$(v_3,b_3)$ and $(v_4,b_4)$, also denoted simply by $(v_3)$ and $(v_4)$, have received by far the greatest attention  (see \cite{v3,BP} and \cite{betten} for enumeration results). Configurations of type $(v_4, b_3)$ (which correspond to graphs of valence $\{3,4\}$)
have also been studied, but much less is know about them (see \cite{gropp}, for example). 
The graphs in Table~2 can therefore be viewed as a contribution to the theory of flag-transitive configurations of type $(v_4, b_3)$.
\medskip

Here are some additional comments:
\medskip

(1) The graph with ID $[14,2]$ is the incidence graph of the point--side incidence structure of a cube; that is,
the vertices of the graph $[14,2]$ can be identified with the the $8$ points and $6$ sides of a cube in such a way that edges correspond to
incident pairs point-side. Hence the automorphism group of the graph $[14,2]$ coincides with the group of symmetries of the cube and is thus
isomorphic to $S_4\times C_2$.
\medskip

(2) The graph with ID $[21,2]$ is the incidence graph of the point--line geometry of the affine plane $\ZZ_3^2$; ie. the 9 vertices of
valence $4$ correspond to the 9 points in $\ZZ_3^2$, the $12$ vertices of valence $3$ correspond to the $12$ lines in $\ZZ_3^2$
and a point is adjacent to a line whenever it lies on the line.
\medskip

(3) The graphs with ID $[28,2]$ and $[28,3]$ are the incidence graphs of  one of $574$ configurations of type $(12_4;16_3)$ (see \cite{gropp}).
The configuration corresponding to the graph $[28,3]$ is the so called {\em Reye configuration}; see for example 
\cite{SerSer}. It is interesting that both configurations can be realised in the plane by straight lines only.
Let us also mention the following interesting construction of the graph $[28,3]$. Let $U=\ZZ_4^2$,  and let 
$W=\{(X, Y) : X \subseteq \ZZ_4, |X|=2, Y \in \{X,\ZZ_4 \setminus X\}\}$. The graph $[28,3]$ can then be viewed
as the graph with vertex set $U \cup W$ and with $(x,y) \in U$ being adjacent to $(X,Y) \in W$ if and only if $x\in X$ and $y\in Y$.
\bigskip

\begin{small}
\begin{center}
Table~2: The list of connected graphs of valence $\{3,4\}$ on at most $100$ vertices  \\
admitting a locally arc-transitive group with a trivial edge kernel 
\bigskip

\begin{tabular}{|c|c|c|c|c|c|c|c|c|}
\hline
ID & girth & diameter & worthy & $(A_v^{\Gamma(v)},A_u^{\Gamma(u)})$ & $(s_v,s_u)$ & $|A_{uv}|$ & $|A_{uv}^{[1]}|$ & comments \\
 \hline \hline
$[7,1]$ & $ 4 $ & $ 2 $ & no & $(S_3,S_4)$ &  $(3,3)$ & 
           $2^{2} \cdot 3$ &  $1$ & $K_{3,4}$ \\ \hline 
$[14,1]$ & $ 4 $ & $ 4 $ & no & $(S_3,D_4)$ &  $(1,2)$ & 
           $2^{4}$ &  $2^{2}$ &  $\SD(K_4)$  \\ \hline 
$[14,2]$ & $ 4 $ & $ 4 $ & yes & $(S_3,D_4)$ &  $(1,2)$ & 
           $2$ &  $1$ & see (1) \\ \hline 
$[21,1]$ & $ 4 $ & $ 4 $ & no & $(S_3,D_4)$ &  $(1,2)$ & 
           $2^{7}$ &  $2^{5}$ &  $\SD(K_{3,3})$ \\ \hline 
$[21,2]$ & $ 6 $ & $ 4 $ & yes & $(S_3,S_4)$ &  $(3,4)$ & 
           $2^{2} \cdot 3$ &  $1$ & see (2) \\ \hline 
$[28,1]$ & $ 4 $ & $ 6 $ & no & $(S_3,D_4)$ &  $(1,2)$ & 
           $2^{8}$ &  $2^{6}$ &  $\SD(Q_3)$ \\ \hline 
$[28,2]$ & $ 6 $ & $ 4 $ & yes & $(S_3,D_4)$ &  $(1,2)$ & 
           $2^{2}$ &  $1$ &  see (3) \\ \hline 
$[28,3]$ & $ 6 $ & $ 4 $ & yes & $(S_3,S_4)$ &  $(3,3)$ & 
           $2^{2} \cdot 3$ &  $1$ & Reye; see (3)\\ \hline 
$[35,1]$ & $ 4 $ & $ 6 $ & no & $(S_3,D_4)$ &  $(1,2)$ & 
           $2^{11}$ &  $2^{9}$ &  $\SD(\Pet])$ \\ \hline 
$[35,2]$ & $ 6 $ & $ 6 $ & yes & $(S_3,S_4)$ &  $(3,3)$ & 
           $2^{2} \cdot 3$ &  $1$ & \\ \hline 
$[42,1]$ & $ 6 $ & $ 4 $ & yes & $(S_3,D_4)$ &  $(1,2)$ & 
           $2^{4}$ &  $2^{2}$ & \\ \hline 
$[42,2]$ & $ 6 $ & $ 4 $ & yes & $(S_3,D_4)$ &  $(1,2)$ & 
           $2^{4}$ &  $2^{2}$ & \\ \hline 
$[49,1]$ & $ 4 $ & $ 6 $ & no & $(S_3,D_4)$ &  $(1,2)$ & 
           $2^{16}$ &  $2^{14}$ &  $\SD(F[14,1])$ \\ \hline 
$[49,2]$ & $ 6 $ & $ 6 $ & yes & $(S_3,D_4)$ &  $(1,2)$ & 
           $2$ &  $1$ & \\ \hline 
$[49,3]$ & $ 6 $ & $ 5 $ & yes & $(S_3,D_4)$ &  $(1,2)$ & 
           $2^{2}$ &  $1$ & \\ \hline 
$[56,1]$ & $ 4 $ & $ 8 $ & no & $(S_3,D_4)$ &  $(1,2)$ & 
           $2^{16}$ &  $2^{14}$ &  $\SD(F[16,1])$ \\ \hline 
$[56,2]$ & $ 6 $ & $ 6 $ & yes & $(S_3,D_4)$ &  $(1,2)$ & 
           $2$ &  $1$ & \\ \hline 
$[56,3]$ & $ 6 $ & $ 6 $ & yes & $(S_3,D_4)$ &  $(1,2)$ & 
           $2^{4}$ &  $2^{2}$ & \\ \hline 
$[56,4]$ & $ 8 $ & $ 6 $ & yes & $(S_3,S_4)$ &  $(3,3)$ & 
           $2^{2} \cdot 3$ &  $1$ & \\ \hline 
$[56,5]$ & $ 6 $ & $ 6 $ & yes & $(S_3,D_4)$ &  $(1,2)$ & 
           $2^{2}$ &  $1$ & \\ \hline      
$[63,1]$ & $ 4 $ & $ 8 $ & no & $(S_3,D_4)$ &  $(1,2)$ & 
           $2^{19}$ &  $2^{17}$ &  $\SD(F[18,1])$ \\ \hline 
$[63,2]$ & $ 8 $ & $ 6 $ & yes & $(S_3,S_4)$ &  $(3,3)$ & 
           $2^{2} \cdot 3$ &  $1$ & \\ \hline 
$[63,3]$ & $ 8 $ & $ 6 $ & yes & $(S_3,D_4)$ &  $(1,2)$ & 
           $2$ &  $1$ & \\ \hline 
$[63,4]$ & $ 8 $ & $ 6 $ & yes & $(S_3,S_4)$ &  $(3,4)$ & 
           $2^{2} \cdot 3$ &  $1$ & \\ \hline 
$[70,1]$ & $ 4 $ & $ 10 $ & no & $(S_3,D_4)$ &  $(1,2)$ & 
           $2^{20}$ &  $2^{18}$ &  $\SD(F[20,1])$ \\ \hline 
$[70,2]$ & $ 4 $ & $ 10 $ & no & $(S_3,D_4)$ &  $(1,2)$ & 
           $2^{21}$ &  $2^{19}$ &  $\SD(F[20,2])$ \\ \hline 
$[70,3]$ & $ 8 $ & $ 6 $ & yes & $(S_3,S_4)$ &  $(3,3)$ & 
           $2^{2} \cdot 3$ &  $1$ & \\ \hline 
$[84,1]$ & $ 4 $ & $ 8 $ & no & $(S_3,D_4)$ &  $(1,2)$ & 
           $2^{24}$ &  $2^{22}$ &  $\SD(F[24,1])$ \\ \hline 
$[84,2]$ & $ 8 $ & $ 6 $ & yes & $(S_3,D_4)$ &  $(1,2)$ & 
           $2$ &  $1$ & \\ \hline 
$[84,3]$ & $ 8 $ & $ 6 $ & yes & $(S_3,\ZZ_2^2)$ &  $(1,2)$ & 
           $2$ &  $1$ & \\ \hline 
$[84,4]$ & $ 6 $ & $ 8 $ & yes & $(S_3,D_4)$ &  $(1,2)$ & 
           $2^{3}$ &  $2$ & \\ \hline 
$[91,1]$ & $ 4 $ & $ 10 $ & no & $(C_3,D_4)$ &  $(1,1)$ & 
           $2^{25}$ &  $2^{24}$ &  $\SD(F[26,1])$ \\ \hline 
$[91,2]$ & $ 8 $ & $ 6 $ & yes & $(C_3,C_4)$ &  $(1,1)$ & 
           $1$ &  $1$ & \\ \hline 
$[98,1]$ & $ 4 $ & $ 10 $ & no & $(S_3,D_4)$ &  $(1,2)$ & 
           $2^{29}$ &  $2^{27}$ &  $\SD(F[28,1])$ \\ \hline 
$[98,2]$ & $ 6 $ & $ 6 $ & yes & $(S_3,D_4)$ &  $(1,2)$ & 
           $2$ &  $1$ & \\ \hline 
$[98,3]$ & $ 6 $ & $ 8 $ & yes & $(S_3,D_4)$ &  $(1,2)$ & 
           $2$ &  $1$ & \\ \hline 
$[98,4]$ & $ 8 $ & $ 6 $ & yes & $(S_3,D_4)$ &  $(1,2)$ & 
           $2$ &  $1$ & \\ \hline 
$[98,5]$ & $ 6 $ & $ 6 $ & yes & $(S_3,D_4)$ &  $(1,2)$ & 
           $2^{9}$ &  $2^{7}$ & \\ \hline 
$[98,6]$ & $ 6 $ & $ 6 $ & yes & $(S_3,D_4)$ &  $(1,2)$ & 
           $2^{9}$ &  $2^{7}$ & \\ \hline 
$[98,7]$ & $ 6 $ & $ 6 $ & yes & $(S_3,D_4)$ &  $(1,2)$ & 
           $2$ &  $1$ & \\ \hline 
$[98,8]$ & $ 8 $ & $ 6 $ & yes & $(S_3,D_4)$ &  $(1,2)$ & 
           $2^{2}$ &  $1$ & \\ \hline 
$[98,9]$ & $ 6 $ & $ 6 $ & yes & $(S_3,\ZZ_2^2)$ &  $(1,2)$ & 
           $2$ &  $1$ & \\ \hline
\end{tabular}
\end{center}
\end{small}


\end{document}